\documentclass[a4paper, 12pt]{article}
\usepackage{graphics}
\usepackage{ifthen}
\usepackage{amsmath,amsfonts,amstext,amscd,bezier,amssymb,a4,enumerate,epsfig,psfrag,subfigure}
\usepackage[latin1]{inputenc}
\usepackage{graphicx}
\usepackage{amsthm,amsopn}
\usepackage{bm}
\usepackage{amssymb}
\usepackage[T1]{fontenc}
\usepackage{mathrsfs}

\usepackage{color}

\everymath={\displaystyle}

\newtheorem*{theo*} {Theorem}

\renewcommand{\qed}{\hfill \mbox{\raggedright \rule{.1in}{.1in}}}

\begin{document}
\centerline{\bf EIGENVECTORS FROM EIGENVALUES REVISITED}

\bigskip
\centerline{ Carlos Tomei}

\centerline{\small Depto. de Matemática, PUC-Rio,  Brazil, carlos.tomei@gmail.com }

\bigskip
Recently, Denton, Parke, Tao and Zhang posted an article in ArXiv [DPTZ] presenting a formula for the  eigenvectors of a Hermitian $n \times n$ matrix $A$ in terms of the eigenvalues of $A$ and the matrices $M_i$ obtained from removing the i-th row and column. It turns out that the formula is familiar in different contexts. It is equation (4.9) in [CG], for example, where the first coordinates of the eigenvectors (up to phase) are expressed from eigenvalues of $A$ and $M_1$. The key ingredient is a meromorphic function $f(\lambda)$ which appears also  in a classical proof of the spectral theorem for operators in Hilbert space [D].

To simplify notation, as in their proofs, denote by $M$ the $(n-1) \times (n-1)$ matrix obtained from deletion of the first column and row of $A$. 

Let $v_i$ be the first coordinate of a normalized eigenvector associated with the eigenvalue $\lambda_i \in \sigma(A)$. Clearly, $|v_i|^2$ is well defined.

\begin{theo*}
$ \ \  |v_i|^2 \ \Pi_{k=1; k \ne i}^{n} \ \big( \lambda_i(A) - \lambda_k(A) \big) \ = \ \Pi_{k=1}^{n-1} \ \big(\lambda_i(A) - \lambda_k(M) \big) \ $.
\end{theo*}

\begin{proof} Yet another simplification: we prove  a generic case, where $A$ and $M$ have $2n-1$ distinct eigenvalues --- for the general case, use density and take limits of the equation in the statement of the theorem. From the spectral theorem, $A = V^\ast \Lambda V$, where the rows of the unitary matrix $V$ are normalized eigenvectors of $A$, and $\Lambda$ is a real, diagonal matrix. By Cramer's rule,

\[  f(\lambda) = \langle e_1, (A - \lambda I)^{-1} e_1 \rangle = \frac{ \det(M - \lambda I)}{\det(A - \lambda I)} = \frac{ \Pi_{k=1}^{n-1} \ \big( \lambda_k(M) - \lambda \big)}{\Pi_{k=1}^{n} \ \big(  \lambda_k(A) - \lambda \big)} \eqno(\ast)\]
and
\[ f(\lambda) =
\langle V e_1, (\Lambda - \lambda I)^{-1} V e_1 \rangle
 = \sum_{i=1}^n \frac{|v_i|^2}{\lambda_i(A) - \lambda} \ , \]
The result follows by computing the residues $|v_i|^2$ in equation $(\ast)$. \qed
\end{proof}

\bigskip
\noindent
[DPTZ] P. Denton, S. Parke, T. Tao, X. Zhang, {\it Eigenvectors from eigenvalues}, ArXiv:1908.03795v1

\noindent
[CG] M.Chu, E. Golub, {\it Structured inverse eigenvalue problems}, Acta Numerica (2002), 1-71.

\noindent
[D] P. Deift, {\it Orthogonal Polynomials and Random Matrices: a Riemann-Hilbert Approach}, Courant Lecture Notes 3 (2000), AMS.

\end{document}